\documentclass[12pt,a4paper,reqno]{amsart}
\usepackage{graphicx}
\usepackage{amssymb,amsmath,amsthm,mathrsfs,enumerate}
\usepackage{color, comment, bbm}
\usepackage[running]{lineno}
\newtheorem{definition}{Definition}[section]
\newtheorem{proposition}[definition]{Proposition}
\newtheorem{theorem}[definition]{Theorem}
\newtheorem{corollary}[definition]{Corollary}
\newtheorem{lemma}[definition]{Lemma}

\theoremstyle{definition}
\newtheorem{remark}[definition]{Remark}

\newtheorem{example}[definition]{Example}

\newcommand{\I}{\operatorname{Id}}

\newcommand{\ud}{\mathrm{d}}

\newcommand{\1}{\mathbbm{1}}

\title[Regularity of SPDEs in metric measure spaces]{Regularity of the solutions to SPDEs in metric measure spaces}

\author{Elena Issoglio}
\address[E. Issoglio]{University of Leeds, School of Mathematics, Leeds, LS2 9JT, UK}
\email{e.issoglio@leeds.ac.uk}
\author{Martina Z\"ahle}
\address[M. Z\"ahle]{Friedrich Schiller University,
             Institute of Mathematics, 07743 Jena, Germany}

\email{martina.zaehle@uni-jena.de}

\thanks{\center \vspace{-10pt} \fbox{\parbox{\textwidth}{The final publication is available at Springer in \emph{Stoch PDE: Anal Comp} 
  via  http://dx.doi.org/10.1007/s40072-015-0048-8}}}

\keywords{Stochastic nonlinear PDEs, regularity of solutions, pathwise solutions, semigroups, metric measure spaces, fractals}
\subjclass[2010]{Primary: 60H15, Secondary: 31E05, 35K55, 35R60, 28A80}

\begin{document}

\begin{abstract}
In this paper we study the regularity of non-linear parabolic PDEs and stochastic PDEs on metric measure spaces admitting heat kernel estimates. In particular we consider mild function solutions to abstract Cauchy problems and show that the unique solution is H\"older continuous in time with values in a suitable fractional Sobolev space. As this analysis is done via a-priori estimates, we can apply this result to stochastic PDEs on metric measure spaces and solve the equation in a pathwise sense for almost all paths. The main example of noise term is of fractional Brownian type and the metric measure spaces can be classical as well as given by various fractal structures. The whole approach is low dimensional and works for spectral dimensions less than 4.
\end{abstract}

\maketitle

% Running line numbers:
%\linenumbers
%\setlength\linenumbersep{5pt}
%\renewcommand\linenumberfont{\normalfont\tiny\sffamily\color{green}}

%%%%%%%%%%%%%%%%%%%%%%%%%%%%%%%%%%%%%%%%%%%%%%%%%%%%%%%%%%%%%%%%%%

%\tableofcontents

\section{Introduction}

In this paper the following non-linear Cauchy problem
\begin{equation}\label{eq: introduction: Cauchy pb}
\begin{cases}
\frac{\partial u}{\partial t} = -A u + F(u) + G(u)\cdot \dot{z}, & t\in(0,t_0]\\
u(0)=f
\end{cases}
\end{equation}
is considered on $\sigma $-finite metric measure spaces $(X,\mu ,d)$. Here $t_0>0$ is arbitrary, $-A$ is the generator of a {\it Markovian strongly continuous symmetric semigroup} $\{T(t), t\geq0\}$ on  $L_2(\mu)$, $F$ and $G$ are sufficiently regular functions. The term  $\dot{z}$ denotes a fractional space-time perturbation which will be made more precise later on. In the case of linear spaces it can be interpreted as a formal time derivative of a spatial distribution. Solutions to \eqref{eq: introduction: Cauchy pb} are considered in the mild form, formally given by
\begin{multline}\label{eq: introduction: mild sol}
u(t)= T(t)f + \int_0^t T(t-s)F(u(s)) \ud s + \int_0^t T(t-s)G(u(s)) \ud z( s).
\end{multline}
This formulation is, in a first place, only formal. We will give it a proper mathematical meaning as it is done in \cite{HZ12}, in particular the last term involving the noise $z$ is defined by means of fractional derivatives and it is shown to be indeed well defined using the notion of pointwise product of functions and ``distributions''. The spaces which we will use to describe the space-regularity of the solution \eqref{eq: introduction: mild sol} are fractional Sobolev spaces defined on metric measure spaces by means of the associated semigroups.

The main aim of this paper is to show that the mild solution $u$ of the Cauchy problem \eqref{eq: introduction: Cauchy pb} given by \eqref{eq: introduction: mild sol} is  $\gamma$-H\"older continuous in time with respect to the $H^\delta(\mu)$-norm in space (Theorem \ref{thm: gamma holder regularity}). This result is achieved under the same assumptions as in Theorem 1.2 in \cite{HZ12}, where it was shown that a unique solution exists and belongs to the space $W^\gamma([0,t_0], H^\delta_\infty(\mu))$. Moreover, under slightly stricter conditions, we can show that the solution in fact belongs to any space $W^\gamma([0,t_0], H^\delta_\infty(\mu))\cap C^\gamma([0,t_0], H^\delta(\mu))$ for all $\gamma$ and $\delta$ smaller than certain parameters determined by the regularity properties of the distributional noise $z$ and the initial function $f$ (Corollary \ref{cor: extra holder regularity}). In the case of noises of fractional Brownian type these parameters are determined my means of the Hurst exponents in space and in time. Note that the fractional time regularity is always greater than $1/2$.\\ In Remark \ref{rem: low spectral dim} we outline an extension of the results to an appropriate parameter condition for the case of spectral dimensions $\le 1$, which has not been considered in \cite{HZ12}. In particular, white noise in space can be treated in this low dimensional case.

Deterministic elliptic equations and some parabolic equations without noises on classes of fractals have been studied, e.g. in \cite{Ba98,Fa99,FH01,GHL03,St06,FHS12}.\\
Abstract problems with Brownian and fractional Brownian noises have been considered in many papers with various approaches, in particular, in \cite{Wa86,dPZ92,GA99,MN03,TTV03,GLT06,PR07,FKN11}. None of these covers the results of the present paper. \\
Some relationships have been discussed in \cite{HZ12}, see also \cite{HIZ14}. To these two references, we only add a brief  comparison between the present paper and the rough path approach developed in recent years. For example in a series of papers \cite{HV11, HW13, HMW14} the rough path approach has been applied to study  a stochastic Burger-type equation with multiplicative white noise. % The aim was to give a meaning to the equation, study existence and uniqueness of the solution, provide numerical and analytical approximations and study their convergence.
Even though the equation considered there is of a different kind than the one studied in this paper (the setting is Euclidean and not metric measure space, the noise is white and not coloured and there is a non-linear product term of the form $g(u)\partial_x u$ which we do not have), it is interesting to notice that one of the main difficulties is to give a meaning to the solution, in particular to the non-linear product term $g(u)\partial_x u$, and this is done using the notion of \emph{paraproduct}.

\section{Preliminaries}
\subsection{Semigroups and potential spaces}\label{semigroups_potential spaces}

Throughout the paper we use the letter  $c$ for a general finite constant which might change value from line to line.\\
We first recall some basic notions and relationships which are known from the literature.
In the case of metric measure spaces the analogues of the classical fractional Sobolev (or Bessel potential) spaces in the literature are introduced by means of the given semigroup $\{T(t), t\geq0\}$, i.e., of its generator $-A$:

The \emph{generalized Bessel potential operator} on $L_2(\mu)$ is defined for $\sigma\geq 0$ as
\[
J^\sigma (\mu):= (A+\I)^{-\sigma/2}.
\]
To each operator there corresponds a \emph{potential space}  defined as
\[
H^\sigma(\mu):= J^\sigma(L_2(\mu))
\]
and equipped with the norm $\|u\|_{H^\sigma(\mu)}:= \| u\|_{L_2(\mu)}+\|A^{\sigma/2} u\|_{L_2(\mu)}$, which is equivalent to $\|(A+\I)^{\sigma/2}u\|_{L_2(\mu)}$. In fact these spaces correspond to the domains of fractional powers of $A$, i.e., $D((A+\I)^{\sigma/2}) =D(A^{\sigma/2})= H^{\sigma}(\mu)$. In particular, for any $\alpha\geq0$ the operator $J^\alpha$ acts as an isomorphisms between $H^\sigma(\mu)$ and $H^{\alpha + \sigma}(\mu)$. Analogously one can define the potential spaces corresponding to the generators $-A_p, 1<p<\infty$, of Markovian semigroups on $L_p(\mu)$. They are denoted by $H_p^\sigma(\mu)$ and clearly $H_2^\sigma(\mu)=H^\sigma(\mu)$. We will also consider the spaces
$$ H^\sigma_\infty(\mu) := H^\sigma(\mu)\cap L_\infty(\mu)$$
 normed by $\|\cdot\|_{H^\sigma(\mu)}+\|\cdot\|_\infty$, with slight abuse of notation. Here the norm $\|\cdot\|_\infty$ in $L_\infty(\mu)$ is given by the essential supremum.\\ The dual spaces of $H^\sigma_p(\mu)$ will be used in the sequel:  for $1<p<\infty, \sigma\geq0$ they are denoted by
\[
H^{-\sigma}_{p'}(\mu):= \left(H^\sigma_p(\mu)\right)^*,
\]
where $\frac1p + \frac{1}{p'}=1$. In case $p=2$ we do not write $p$ explicitly. Note that $H^{-\sigma}(\mu)\subseteq \left(H^\sigma_\infty(\mu)\right)^*$ often being a strict inclusion.

 For the regularity in time of the solution we consider the following spaces frequently used in the literature: Let $0<\eta<1$ and $(X, \|\cdot\|_X)$ be a normed space. Then $W^\eta([0, t_0], X)$ denotes the space of functions $v:[0, t_0]\to X$ such that $\|v\|_{\eta, X}<\infty$, where
\[
\|v\|_{\eta, X}:= \sup_{0\leq t\leq t_0} \left(\|v(t)\|_X + \int_0^t \frac{\|v(t)-v(s)\|_X }{(t-s)^{\eta+1}}\ud s \right)
\]
is the norm in $W^\eta([0, t_0], X)$.\\

\vskip3mm
We will use the
{\it short notations} for the following norms:
$$\|\cdot\|_{\alpha, \infty}:= \|\cdot\|_{ H^\alpha_\infty (\mu)}~~\mbox{and}~~ \|\cdot\|_{\alpha}:= \|\cdot\|_{ H^{\alpha} (\mu)}~~\mbox{for each}~~ \alpha\in
\mathbb{R}\, .$$
Then we recall that for $\nu\ge 0$ and $t>0$ the operators $T(t)$ and  $A^\nu$ commute on $D(A^\nu)$ and satisfy the following well-known estimates (see e.g. \cite{Pa83}) for $u,v\in D(A^\nu)$:
\begin{equation} \label{eq: T(t)- extra regularity} \|T(t)v||_\nu\le c  t^{-\frac{\nu}{2}}\|v\|_0\, ,\end{equation}
and  for $0<\nu< 1$,
\begin{equation} \label{eq: T(t)- alpha continuity in L2} \|T(t)u-u\|_0\le c t^\nu\|u\|_{2\nu}\, ,\end{equation}
where $0<t\leq t_0$.\\
The symmetry of the semigroup $\{T(t), t\geq 0\}$ has been used in order to extend it to elements from the dual spaces.
If $w\in H^{-\beta}(\mu)$ then $T(t)w$ is the element of $L_2(\mu)$ determined by the duality relation
$$(v,T(t)w):=(T(t)v,w)\, ,~~v\in L_2(\mu)\, .$$
Then we get
$$|(v,T(t)w)|=|(T(t)v,w)|\le \|T(t)v\|_\beta \|w\|_{-\beta}$$
and hence,
$$\|T(t)w\|_0\le c t^{-\frac{\beta}{2}}\|w\|_{-\beta}$$
in view of \eqref{eq: T(t)- extra regularity}. Applying the latter again and using $T(t)=T(\frac{t}{2})\circ T(\frac{t}{2})$ we infer
\begin{equation}\label{eq: T(t) delta-beta bound}
\|T(t)w\|_\delta\le c t^{-\frac{\delta}{2}-\frac{\beta}{2}}\|w\|_{-\beta}
\end{equation}
for any $\delta,\beta>0$.\\
Similarly one obtains from \eqref{eq: T(t)- alpha continuity in L2}
\begin{equation}\label{eq: T(t)-alpha continuity}
\|T(t)w-w\|_{-\beta-2\nu}\le c t^\nu\|w\|_{-\beta}\, ,
\end{equation}
for any $\beta>0$, $w\in H^{-\beta}(\mu)$ and $0<\nu<1$.\vskip3mm
Note that the constants in the estimates depend on the related parameters.
\vskip3mm
Throughout the paper we make the following standing assumptions which are the same as in \cite{HZ12}:
\begin{description}
 \item[\textbf{Assumption (MMS)}] $(X,d)$ is a locally compact separable metric space. We consider the Borel $\sigma$-field on $X$ and a Radon measure $\mu$ on $(X,d)$.
\item[\textbf{Assumption (HKE($\beta$))}]
The transition kernel $P_t(x, dy)$ associated with the semigroup $T(t), t\geq0$ admits a transition density $P_t(x, dy) = p(t,x,y) \mu(dy)$ which satisfies for almost all $x, y \in X$ the following {\it heat kernel estimate}
\[
t^{-\frac{d_H}{w}} \Phi_1 (t^{-\frac{1}{w}}d(x,y) ) \leq p(t,x,y) \leq t^{-\frac{d_H}{w}} \Phi_2 (t^{-\frac{1}{w}}d (x,y) )
\]

if $0<t<R_0$ for some constants $R_0>0$, $w\ge 2$ and nonnegative bounded decreasing functions $\Phi_i$
on $[0,\infty)$, where $d_H$ is the {\it Hausdorff dimension} of $(X,d)$. For $t\ge R_0$, $$p(t,x,y)\le p_t$$ and $p_t$ decreases in $t$. (In this case the semigroup is {\it ultracontractive}, i.e.,$$ \|T(t)\|_{L_\infty(\mu)}\le p_t\|f\|_{L_2(\mu)}\, ,$$  where $p_t:=c\, t^{-d_S/4}$, if $t<R_0$, and the value $d_S=\frac{2d_H}{w}$ agrees with its {\it spectral dimension}. $w$ is called {\it walk dimension} of the semigroup.) For a given $\beta>0 $ we further assume the {\it integrability condition} $$\int_0^\infty s^{d_H+\beta w/2-1}\Phi_2(s) \mathrm d s<\infty\, .\footnote{In \cite{HZ12} the $w$ is missing in the exponent.}$$
\end{description}
Heat kernels of this type have been studied in Grigor'yan and Kumagai \cite{GK08} and related references therein. Further relationships are presented in the recent survey \cite{GHL14} of Grigor'yan, Hu and Lau.
\vskip3mm
In order to make the integral in \eqref{eq: introduction: mild sol} precise we need {\it pointwise products} of functions and dual elements from the potential spaces. In \cite{HZ12} the following is proved which also extends related results for the Euclidean case.
\begin{proposition} \label{prop: hinz zaehle_paraproducts} \cite[Corollary 4.1]{HZ12}\\
Suppose (MMS) and (HKE($\beta$))
for $0<\beta <\delta<\min(\frac{d_S}{2}, 1)$. Then for $q=\frac{d_S}{\delta}$ the product $gh$ of  $g\in H^\delta(\mu)$ and $h\in H^{-\beta}_q(\mu)$ is well defined in $H^{-\beta}(\mu)$ by the duality relation $(f,gh):= (fg,h)\, ,~ f\in H^\delta(\mu)$, and the following estimate holds true:
$$\|gh\|_{-\beta}\leq c \|g\|_\delta \|h\|_{H^{-\beta}_q(\mu)}\, .$$
\end{proposition}
\vskip3mm

\subsection{The integral equation and mild solution}

A rigorous definition for the integral and a contraction principle for the solution to equation \eqref{eq: introduction: mild sol} are given in \cite{HZ12}  by means of fractional calculus in Banach spaces, in particular, under the following additional conditions.
\vskip3mm
\begin{description}
\item[\textbf{Assumption (FG)}]
The nonlinear functions $F$ and $G$ are such that $F\in C^1(\mathbb R^n)$, $F(0)=0$ and $F$ has bounded Lipschitz derivative $F'$ and $G\in C^2(\mathbb R^n), G(0)=0$ and $G$ has bounded Lipschitz second derivative $G^{''}$.
\end{description}
\vskip3mm
For the parameters we here consider the case II from \cite{HZ12}.
\begin{description}
\item[\textbf{Assumption (P)}]
$0<\alpha<\gamma$, $0<\beta<\delta<\min(\frac{d_S}{2},1)$, $\gamma<1-\alpha-\frac{\beta}{2}-\frac{d_S}{4}$, where $\beta$ and $d_S$ are from (HKE($\beta$)), and $q=\frac{d_S}{\delta}$.
\end{description}
\vskip3mm
We now will briefly summarize the construction.\vskip3mm
If $u\in W^\gamma([0,t_0], H^{\delta}_\infty(\mu) )$ the operator
 $U(t;s): H^{-\beta}_q(\mu)\to H^{\delta}_\infty(\mu) $ is defined as
\begin{equation}\label{eq: U(t;s)}
U(t;s)w:=T(t-s) \left( G(u(s))w \right)
\end{equation}
for $w\in H^{-\beta}_q(\mu)$. Then under the above assumptions on the function $G$  and the parameters (P) for any $0<\eta<\gamma$  the {\it left-sided Weyl-Marchaud fractional derivative} of order $\eta$ is determined by
\[
D^\eta_{0+} U(t;s) := \frac{{\1}_{(0,t)}(s)}{\Gamma(1-\eta)} \left( \frac{U(t;s)}{s^\eta}  + \eta \int_0^s \frac{U(t;s) -U(t;\tau)}{(s-\tau)^{\eta+1}} \mathrm d \tau \right)
\]
as an element of $L_1([0,t],L(H^{-\beta}_q(\mu),H^\delta(\mu))$ (in the sense of Bochner integration). This is shown in \cite[Lemma 5.2, (ii)]{HZ12}\footnote{We remark that there is a typo in \cite[Lemma 5.2]{HZ12}, namely in (ii) and (iii) the right hand side of the main condition on the parameters should read $2-2\eta -\beta$ instead of $2-2\eta -(\beta\vee\frac{d_S}{2})$.}.\\
Let us  now consider the regulated version of $z\in C^{1-\alpha}([0,t_0],H_q^{-\beta})$ on $[0,t]$ given by $z_t(s) := \1_{(0,t)}(s) (z(s)-z(t))$. If additionally $1-\eta<1-\alpha$, which is always possible in view of (P), one can define the {\it right-sided Weyl-Marchaud fractional derivative} of $z_t$ of order $1-\eta$ by
\[
D^{1-\eta}_{t-} z_t(s) := \frac{(-1)^{1-\eta}\1_{(0,t)}(s)}{\Gamma(\eta)} \left( \frac{z(s)-z(t)}{(t-s)^{1-\eta}}  + (1-\eta) \int_s^t \frac{z(s) -z(\tau)}{(\tau-s)^{(1-\eta)+1}} \mathrm d \tau \right)
\]
as an element of $L_\infty([0,t],H^{-\beta}_q(\mu))$.\\

For more details on these fractional derivative we refer the reader to  \cite{SKM93}, \cite{Za98}, \cite{Za01}, and \cite{HZ09_I} for the Banach space version. They are used for one of the results in \cite{HZ12}:

\begin{proposition} \label{prop: hinz zaehle}
Suppose (MMS), (HKE($\beta$)), (FG), the parameter conditions (P) and $z\in C^{1-\alpha}([0,t_0], H^{-\beta}_q(\mu))$. Then we have the following.
\begin{itemize}
\item[(a)]\cite[Lemma 5.1]{HZ12}
For the operator $U(t;x)=T(t-x)(G(u(x)\cdot)$ as in \eqref{eq: U(t;s)} with $u\in W^\gamma([0, t_0],H_\infty^\delta(\mu))$  the integral $\int_s^t U(t;x)\mathrm d z(x) $ is well defined by
\begin{equation}\label{eq: def of convolution intergal}
\int_s^t U(t;x)\mathrm d z(x) := (-1)^\eta\int_s^t D^\eta_{s+} U(t;x) D^{1-\eta}_{t-} z_t(x) \mathrm d x,
\end{equation}
independently of the choice of $\eta$ with $\eta<\gamma$ and $1-\eta<1-\alpha$. (In particular, the integrand on the right side is a Lebesgue integrable real function.)
\item[(b)] \cite[Theorem 1.2]{HZ12}
For any initial function $f\in H^{2\gamma+\delta+\varepsilon}(\mu)$ with some $\varepsilon>0$ there exists a unique solution $u$ to equation \eqref{eq: introduction: mild sol} for  the definition \eqref{eq: def of convolution intergal} of the integral such that $u\in W^\gamma([0, t_0],H^\delta_\infty(\mu))$.
\end{itemize}
\end{proposition}

\section{The main result}
\subsection{Regularity of the solution}\label{subsc: regularity}

The main results are stated in Theorem \ref{thm: gamma holder regularity} and Corollary \ref{cor: extra holder regularity}.
\vskip3mm
In this section we use {\it short notations} for the following norms: $$\|\cdot\|_L:= \|\cdot\|_{L(H^{-\beta}_q(\mu),H^\delta(\mu))}~~,~~\|\cdot\|_{W^\gamma}:=\|\cdot\|_{W^\gamma([0, t_0], H^\delta_\infty(\mu))}$$
with the specified parameters as in Assumption (P).\\
First recall that under Assumption (FG) the nonlinear operators $F$ and $G$ are bounded from $H^\delta_\infty(\mu)$ into itself (see \cite[Proposition 3.1]{HZ12}).
\begin{lemma}\label{lm: holder regularity of F term}
Suppose (MMS), (HKE($\beta$)), (FG), the parameter conditions (P) and let $z\in C^{1-\alpha}([0,t_0], H^{-\beta}_q(\mu))$. Then there is a positive constant $c$ such that
 \begin{equation} \label{eq: holder regularity of F term}
\left\| \int_s^t T(t-x) F(u(x))\mathrm d x \right\|_{ \delta} \leq  c||u||_{W^\gamma}(t-s).
 \end{equation}
%{\color{red} take out the $\infty$ norm.}
\end{lemma}

\begin{proof}
Since the semigroup $T(t)$ is a contraction on $H^\delta(\mu)$ we have
\begin{align*}
 \left\| \int_s^t T(t-x) F(u(x))\mathrm d x \right\|_{ \delta}
 \leq&  \int_s^t \| T(t-x) F(u(x)) \|_{ \delta}  \mathrm d x \\
 \leq&  \int_s^t \|F(u(x)) \|_{\delta}  \mathrm d x\\
 \leq&  c \int_s^t  \|u(x) \|_{\delta}  \mathrm d x \leq  c \|u\|_{W^\gamma} (t-s),
\end{align*}
where the latter bound follows from the definition of the $W^\gamma$-norm.
\end{proof}

\begin{lemma}\label{lm: epsilon argument}
Under the same conditions as in Lemma \ref{lm: holder regularity of F term} we have
 \begin{align*}
  \| T(t-x)& \left( G(u(x)) \right) - T(t-y) \left( G(u(x)) \right) \|_L \\
  &\leq c ||u||_{W^\gamma} (t-x)^{-\frac\delta2-\frac\beta2-\nu}  (x-y)^\nu
 \end{align*}
 for any $0<\nu<1$.
 \end{lemma}

 \begin{proof}
Let $h$ be an arbitrary element of $H^{-\beta}_q(\mu)$ with $\|h\|_{H^{-\beta}_q(\mu)}\le 1$. Then Proposition \ref{prop: hinz zaehle_paraproducts} implies
 $$\|G(u(x))h\|_{-\beta} \le c \|u\|_{W^\gamma}$$
 uniformly in the time argument $x\in[0,t_0]$ by the mapping property of $G$ and the definition of the $W^\gamma$-norm. Using \eqref{eq: T(t) delta-beta bound}, \eqref{eq: T(t)-alpha continuity} and the last estimate we infer
\begin{align*}
\| T(t-x) & (\mathrm{Id} - T(x-y) ) G(u(x))h \|_{\delta}\\
&\le c (t-x)^{-\delta/2-\beta/2-\nu}  \|(\mathrm{Id} - T(x-y) ) G(u(x))h \|_{-\beta-2\nu}\\
&\le c (t-x)^{-\delta/2-\beta/2-\nu} (x-y)^\nu\|G(u(x))h\|_{-\beta}\\
&\le c (t-x)^{-\delta/2-\beta/2-\nu} (x-y)^\nu\|u\|_{W^\gamma}\, .
\end{align*}

This bound  together with the definition of the $L$-norm
\begin{align*}
 \| T(t-x) & (\mathrm{Id} - T(x-y) ) G(u(x)) \|_{L}\\
 &= \sup_{\|h\|_{H^{-\beta}_q}\leq 1}  \| T(t-x) (\mathrm{Id} - T(x-y) ) \left( G(u(x)) h\right) \|_{\delta}
 \end{align*}
completes the proof.
\end{proof}

\begin{lemma}\label{lm: holder regularity of G term}
Suppose (MMS), (HKE($\beta$)), (FG), the parameter conditions (P) and let $z\in C^{1-\alpha}([0,t_0], H^{-\beta}_q(\mu))$. Then there is a positive constant $c$ such that
\begin{equation}\label{eq: holder regularity of G term}
\left\| \int_s^t T(t-x)  G(u(x)) \mathrm d z(x) \right\|_{ \delta} \leq c (t-s)^\gamma.
 \end{equation}
% {\color{red} take out the $\infty$ norm.}
\end{lemma}

\begin{proof}
Since by assumption $z\in C^{1-\alpha}([0,t_0], H^{-\beta}_q(\mu))$ then for any $\eta$ such that $1-\eta < 1-\alpha$ we have
$$\sup_{t\in[0,t_0]}\sup_{x\in[0,t]} \left\| D^{1-\eta}_{t-} z_t(x) \right\|_{H^{-\beta}_q(\mu)} \leq c<\infty\, .$$  Let us fix $\eta$ throughout the proof as some number slightly bigger than $\alpha$ such that $\alpha<\eta<\gamma$ and at the same time $\gamma<1-\eta-\frac{\delta}{2} - \frac{\beta}{2}$ which is always possible in view of (P). We then get
\begin{align*}
&\left\| \int_s^t U(t;x)\mathrm d z(x) \right\|_{\delta} \\
=&\left\|  \int_s^t D^\eta_{s+} U(t;x) D^{1-\eta}_{t-} z_t(x) \mathrm d x \right\|_{\delta} \\
\leq& \sup_{t\in[0,t_0]}\sup_{x\in[0,t]} \left\|   D^{1-\eta}_{t-} z_t(x) \right\|_{H^{-\beta}_q(\mu)}  \int_s^t \left\|  D^\eta_{s+} U(t;x)  \right\|_{L} \mathrm d x  \\
\leq& c   \int_s^t \left\|  D^\eta_{s+} T(t-x)  G(u(x))  \right\|_{L} \mathrm d x  \\
\le& c \int_s^t \frac{\| T(t-x) G(u(x))  \|_L}{(x-s)^{\eta} }\mathrm d x  \\
&+  c \int_s^t \int_s^x  \frac{\| T(t-x) G(u(x))  - T(t-y)  G(u(y)) \|_L}{(x-y)^{1+\eta} } \mathrm d y \mathrm d x \\
=&: S_1 + S_2.
\end{align*}
Consider $S_1$ first. Using \eqref{eq: T(t) delta-beta bound} and Proposition \ref{prop: hinz zaehle_paraproducts} we obtain
\begin{align*}
 S_1 \leq &c\int_s^t \sup_{\|w\|_{H^{-\beta}_q(\mu)}\leq 1}\frac{\| T(t-x)\left( G(u(x))w\right) \|_{\delta} }{(x-s)^{\eta}}\mathrm d x\\
\leq & c \int_s^t (t-x)^{-\frac\delta2 - \frac\beta2} \sup_{\|w\|_{H^{-\beta}_q(\mu)} \leq 1}\frac{\| G(u(x)) w \|_{-\beta}}{(x-s)^{\eta}}\mathrm d x\\
\leq & c \int_s^t (t-x)^{-\frac\delta2 - \frac\beta2}  (x-s)^{-\eta} \| G(u(x)) \|_{\delta}   \sup_{\|w\|_{H^{-\beta}_q(\mu)}\leq 1} \|w\|_{H^{-\beta}_q(\mu)} \mathrm d x\\
\leq & c \int_s^t (t-x)^{-\frac\delta2 - \frac\beta2}  (x-s)^{-\eta} \|u\|_{W^\gamma}  \mathrm d x\\
\leq & c  (t-s)^{1-\frac\delta2 - \frac\beta2-\eta} \leq c(t-s)^\gamma,
\end{align*}
the latter following from $1-\frac\delta2 - \frac\beta2-\eta>\gamma$ by construction. Moreover the integral is finite since $\gamma>0$.\\
Consider $S_2$. The numerator inside the integral can be bounded as follows
\begin{align}\label{eq: difference of T}
\| T(t-x) &  G(u(x))  - T(t-y)  G(u(y))  \|_L \\ \nonumber
\leq& \| T(t-x)  G(u(x))  - T(t-y)  G(u(x))  \|_L \\ \nonumber
+&\| T(t-y)  G(u(x))  - T(t-y)  G(u(y))  \|_L,
\end{align}
so that we have
\begin{align*}
S_2=& c \int_s^t \int_s^x  \frac{\| T(t-x)  G(u(x))  - T(t-y)  G(u(y))  \|_L}{(x-y)^{1+\eta} } \mathrm d y \mathrm d x \\
 \leq & c \int_s^t \int_s^x  \frac{\| T(t-x)  G(u(x))  - T(t-y)  G(u(x))  \|_L}{(x-y)^{1+\eta} } \mathrm d y \mathrm d x \\
 +& c \int_s^t \int_s^x  \frac{\| T(t-y)  G(u(x))  - T(t-y)  G(u(y))  \|_L}{(x-y)^{1+\eta} } \mathrm d y \mathrm d x\\
 =: & S_3 + S_4.
\end{align*}
Let us consider the term $S_4$ first. We have (with similar computations as for $S_1$)
\begin{align*}
 S_4=&  c \int_s^t \int_s^x  \frac{\| T(t-y)G(u(x)) - T(t-y)  G(u(y))  \|_L}{(x-y)^{1+\eta} } \mathrm d y \mathrm d x\\
 \leq & c \int_s^t \int_s^x \sup_{\|w\|_{H^{-\beta}_q}\leq1} \frac{\| T(t-y) \left( \left[ G(u(x)) -G(u(y))\right]w\right)   \|_{\delta}}{(x-y)^{1+\eta} } \mathrm d y \mathrm d x\\
\leq & c \int_s^t(t-x)^{-\frac\delta2 - \frac\beta2} \int_s^x  \frac{\| G(u(x)) - G(u(y))  \|_{\delta}}{(x-y)^{1+\eta} } \mathrm d y \mathrm d x\\
\intertext{recall that $\eta<\gamma$ by definition of $\eta$, thus}
\leq & c \int_s^t(t-x)^{-\frac\delta2 - \frac\beta2} \| u \|_{W^\gamma}  \mathrm d x\\
\leq & c (t-s)^{1-\frac\delta2 - \frac\beta2}  \leq c (t-s)^\gamma,
\end{align*}
 the latter being true as $\gamma<1-\frac\delta2 - \frac\beta2$ by assumption.
 Regarding the term $S_3$, we apply Lemma \ref{lm: epsilon argument} with $\nu>\eta$ to the numerator inside the integral of  $S_3$, so that $S_3$ can be bounded by
\begin{align*}
S_3 & \leq c \int_s^t \int_s^x  \frac{\| T(t-x)  G(u(x))  - T(t-y)  G(u(x))  \|_L}{(x-y)^{1+\eta} } \mathrm d y \mathrm d x \\
 & \leq c \int_s^t \int_s^x  \frac{(t-x)^{-\frac\delta2 - \frac\beta2-\nu}  (x-y)^\nu}{(x-y)^{1+\eta} } \mathrm d y \mathrm d x \\
 & \leq c \int_s^t (t-x)^{-\frac\delta2 - \frac\beta2-\nu} \int_s^x  (x-y)^{\nu -1-\eta}  \mathrm d y \mathrm d x \\
 & \leq c \int_s^t (t-x)^{-\frac\delta2 - \frac\beta2-\nu}   (x-s)^{\nu-\eta} \mathrm d x \\
 & \leq c (t-s)^{1-\frac\delta2 - \frac\beta2-\nu+\nu - \eta} \leq c (t-s)^\gamma,
\end{align*}
the latter bound being true as  $\gamma<1-\eta-\frac\delta2 - \frac\beta2$. The proof is complete.
\end{proof}

We are now about to state and prove the main regularity property of the solution $u$ under the Assumptions (MMS), (HKE($\beta$)), (FG) and (P).

\begin{theorem}\label{thm: gamma holder regularity}
Suppose (MMS), (HKE($\beta$)) and (FG).  Let $0<\alpha<\gamma$ , $0<\beta < \delta<\min(\frac{d_S}{2},1)$ and $\gamma<1-\alpha-\frac{\beta}{2} - \frac{d_S}{4}$ and set $q=\frac{d_S}{\delta}$. If $z\in C^{1-\alpha}([0,t_0], H^{-\beta}_q(\mu))$ and  the initial condition $f$ is an element of $H^{\delta+2\gamma+\varepsilon}(\mu)$ for some $\varepsilon>0$, then the unique solution $u \in W^\gamma([0, t_0], H^\delta_\infty(\mu))$ for  \eqref{eq: introduction: mild sol} is also an element of $C^{\gamma}([0, t_0], H^{\delta}(\mu))$.
\end{theorem}
\begin{remark}
For $u \in W^\gamma([0, t_0], H^\delta_\infty(\mu))$ the second integral in equation \eqref{eq: introduction: mild sol} has to be interpreted as above. For $u$ considered as element of $C^{\gamma}([0, t_0], H^{\delta}(\mu))$ this integral agrees with the corresponding Riemann-Stieltjes integral with values in the Banach space $H^{\delta}(\mu)$. The latter has been used in Gubinelly, Lejay and Tindel \cite{GLT06} in an abstract setting in the sense of Young.
\end{remark}
\begin{proof}
Let $0\leq s<t\leq t_0$.  We consider the solution at time $t$ as the evolution of  $u$  according to \eqref{eq: introduction: Cauchy pb} with initial condition at time $s$ being $u(s)$, that is
\[
u(t)= T(t-s)u(s) + \int_s^t T(t-r)F(u(r)) \ud r + \int_s^t T(t-r)G(u(r)) \ud z( r),
\]
so that
\begin{align}\label{eq: difference ut-us}
u(t)-u(s) =&( T(t-s)- \I) u(s) \nonumber\\
+& \int_s^t T(t-r)F(u(r)) \ud r + \int_s^t T(t-r)G(u(r)) \ud z( r)
\end{align}
The $H^\delta(\mu)$-norms of the two integrals are bounded by $ c (t-s)$ and $ c (t-s)^\gamma$ according to Lemma \ref{lm: holder regularity of F term} and Lemma \ref{lm: holder regularity of G term}, respectively.  If we show that $\|u(t)\|_{\delta+2\gamma}<c$ uniformly in $t$ for $t\in[0, t_0]$, then the $H^\delta(\mu)$-norm of the term involving the initial condition can be easily bounded. In fact note that in this case $(A+\I)^{\frac\delta2}u\in D(A^{\gamma}) $ thus  we can apply \eqref{eq: T(t)- alpha continuity in L2} to get
\begin{align*}
 \|(T(t-s)- \I)u(s)\|_{\delta} \leq & c \|(A+\I)^{\delta/2}(T(t-s)- \I)u(s)\|_{0}\\
 \leq & c \|(T(t-s)- \I)(A+\I)^{\delta/2 }u(s)\|_{0}\\
 \leq & c (t-s)^\gamma \|(A+\I)^{\delta/2} u(s)\|_{2\gamma} \\
 \leq & c (t-s)^\gamma \|u(s)\|_{\delta+2\gamma} \leq c (t-s)^\gamma
\end{align*}
as wanted. It remains to prove that $\|u(t)\|_{\delta+2\gamma}<c$ uniformly in $t$ for $t\in[0, t_0]$. Recall that

\[u(t) = T(t)f + \int_0^t T(t-r)F(u(r)) \ud r + \int_0^t T(t-r)G(u(r)) \ud z( r) ,\]
thus
\begin{align*}
 \|u(t)\|_{\delta+2\gamma} &\leq  \|T(t)f\|_{\delta+2\gamma} \\
 &+  \int_0^t\|  T(t-r)F(u(r))\|_{\delta+2\gamma}  \ud r \\
 &+  \int_0^t\| D^\eta_{0+} U(t;r) D^{1-\eta}_{t-} z_t(r) \|_{\delta+2\gamma} \ud r\\
 =: S_1 + S_2  + S_3.
 \end{align*}
The term $S_1$ is easily bounded by $c\|f\|_{\delta+2\gamma}$. The term $S_2$ is bounded recalling that $\|T(t-r) F(u(r))\|_{\delta+2\gamma} \leq (t-r)^{-\gamma} \|F(u(r))\|_\delta$ because of the smoothing action of the semigroup. Thus $S_2\leq c t^{1-\gamma}$. The last term can be treated in a similar way as the proof of Lemma \ref{lm: holder regularity of G term} with the difference that the space $H^{\delta}(\mu)$ is replaced by $H^{\delta+2\gamma}(\mu)$. All computations for the $H^{\delta+2\gamma}(\mu)$-norm term carry out in the same way%the $L_\infty(\mu)$-norm term is not needed here)
, except that the exponent $-\frac\delta2-\frac\beta2$ is replaced by $-\frac\delta2-\frac\beta2 -\gamma$ so that $S_3\leq c t^{1-\eta-\frac\delta2-\frac\beta2-\gamma}$ with $1-\eta-\frac\delta2-\frac\beta2-\gamma>0$ by construction of $\eta$.  Clipping the result together we have
\[
\|u(t)\|_{\delta+2\gamma} \leq c+ ct^{1-\gamma} + c t^{1-\eta-\frac\delta2-\frac\beta2-\gamma},
\]
and finally  taking the supremum over $t\in[0,t_0]$ we get the uniform bound.
\end{proof}

With slightly more restrictive assumptions on the noise we can show that the unique solution $u$ belongs to the spaces $W^\gamma([0,t_0], H^\delta_\infty(\mu))$, and thus to $C^\gamma([0,t_0], H^\delta(\mu))$, for all $(\gamma,\delta)$ such that $0<\gamma<1-\alpha-\frac{\beta}{2}-\frac{d_S}{4}$ and $\beta<\delta<\min(\frac{d_S}{2},1)$.

\begin{corollary}\label{cor: extra holder regularity}
Suppose (MMS), (HKE($\beta$)) and (FG).
\begin{itemize}
 \item[(a)] Let $0<\alpha<\frac{1}{2}$ and $0<\beta<\min(\frac{d_S}{2},1-2\alpha, 2(1-\alpha)-\frac{d_S}{2})$ be given. Suppose that $z\in C^{1-\alpha}([0,t_0],H_q^{-\beta}(\mu))$ for any $1<q<\frac{d_S}{\beta}$ and $f\in H^{2(1-\alpha)-\beta}(\mu)$. Then for any
$\beta<\delta<\min(\frac{d_S}{2},1)$ and $0<\gamma<1-\alpha-\frac{\beta}{2}-\frac{d_S}{4}$
 Equation \eqref{eq: introduction: mild sol} has a unique solution in the space $W^\gamma([0,t_0], H^\delta_\infty(\mu))$ and hence, it has a unique solution belonging to all these spaces.
 \item[(b)]  Moreover, this solution is an element of $C^\gamma([0,t_0], H^\delta(\mu))$ for any $\gamma$ and $\delta$ as before.

\end{itemize}
\end{corollary}

 \begin{proof}
Part (a). Take $\delta$ and $\gamma$ as in the assumption. Then  $2(1-\alpha) -\beta > 2\gamma + \frac{d_S}{2}>2\gamma + \delta + \varepsilon$ for some $\varepsilon>0$ implies that  $f\in H^{2\gamma + \delta+\varepsilon}(\mu)$. Moreover $H^{-\beta}_{q}(\mu) \subset H^{-\beta}_{\frac{d_S}{\delta}}(\mu)$ for $1<q<\frac{d_S}{\beta}$ since $\beta<\delta<d_S$, thus $z\in C^{1-\alpha}([0,t_0], H^{-\beta}_{\frac{d_S}{\delta}}(\mu))$. Then we are under the assumptions of Theorem \ref{thm: gamma holder regularity} and thus there exists a unique solution to \eqref{eq: introduction: Cauchy pb} which belongs to  $ W^\gamma([0,t_0], H^\delta_\infty(\mu))$. Because of the embedding of the spaces involved, clearly $u\in W^{\gamma'}([0,t_0], H^{\delta'}_\infty(\mu))$ for any $0<\delta'\le\delta$ and $0<\gamma'\le\gamma$, too. We also know that for $\delta'$ and $\gamma'$ satisfying the assumptions there exists a unique solution $u'$ to \eqref{eq: introduction: Cauchy pb} which is in $W^{\gamma'}([0,t_0], H^{\delta'}_\infty(
\mu))$. As the initial condition $f$ and the noise term $z$ are the same, then by uniqueness we must have $u=u'$ in the larger space $ W^{\gamma'}([0,t_0], H^{\delta'}_\infty(\mu))$.

Part (b). It follows directly form part (a) and Theorem \ref{thm: gamma holder regularity}% choosing $\delta$ large enough such that $\beta<\delta$
. Obviously, due to the embedding of the fractional Sobolev spaces $H^{\delta}(\mu)\subset H^{\delta'}(\mu)$ for $\delta'<\delta$ and of the H\"older spaces $C^{\gamma}\subset C^{\gamma'}$ for $\gamma'<\gamma$, we have that $u\in C^{\gamma'}([0, t_0], H^{\delta'}(\mu))$ for all $0<\delta'<\delta$ and $0<\gamma'<\gamma$.
\end{proof}
\begin{remark}\label{rem: low spectral dim}
The parameter condition (P), in particular $\delta<d_S/2$, and the integrability condition on the function $\Phi_2$ in the heat kernel estimate (HKE) have been used only for the product estimate in Proposition \ref{prop: hinz zaehle_paraproducts}. An analysis of the proofs, in particular those of \cite{HZ12}, shows that the assertions of Theorem \ref{thm: gamma holder regularity} and Corollary \ref{cor: extra holder regularity} remain valid under the parameter condition
\begin{equation}\label{eq: parameter low spectral dim}
0<\frac{d_S}{2}\le \beta\le\delta<1\, ,~~0<\alpha<\gamma<1-\alpha-\frac{\beta}{2}-\frac{\delta}{2}
\end{equation}
provided $z\in C^{1-\alpha}([0,t_0], H^{-\beta}_q(\mu))$ for some $q>2$ such that the multiplication property
\begin{equation}\label{eq: product assumption}
\|vz\|_{-\beta}\le c\|v\|_\delta\, \|z\|_{H_q^{-\beta}(\mu)}
\end{equation}
holds true for any $v\in H^\delta(\mu)$.\\
Then one obtains a complement to the former assertions for the case $d_S\le 1$: Here we can choose $\frac{1}{2}\le \beta<1-2\alpha$, which implies $1-\alpha>\frac{3}{4}$, in order to get a H\"older continuous solution to equation \eqref{eq: introduction: mild sol}. In particular, in the Gaussian setting (see below) such a $\dot{z}$ may be interpreted as a noise white in space and coloured in time.
\end{remark}

\subsection{Applications and extensions}
With similar techniques as in the previous section it is possible to treat equation \eqref{eq: introduction: Cauchy pb} with $A$ replaced by a fractional power of $A$, that is for instance $A^\theta$ for $0<\theta\le1$. This is done in \cite{HZ12}. In this case the semigroup associated with the fractional power $-A^\theta$ will be the subordinated semigroup $T^{(\theta)}(t)$ and the power $\theta$ must be taken into account in Assumption (P) and all following theorems accordingly. Then similar regularity properties as shown for the case $\theta=1$ follow. These results go in the direction of \cite{RZZ14} where the authors study a stochastic equation with fractional dissipation (that is with a term $A^\theta, 0<\theta\leq 1$) but with the difference that in the present paper  the noise is coloured, whereas in \cite{RZZ14} the noise is white in time and solutions  are strong and local. \\

Clearly, the special case of linear $F$ and $G$ can be considered. In this case the $L_\infty$-boundedness of the solution is no longer needed (see \cite[Theorem 1.3]{HZ12}) and the conditions on the parameters are weaker, in particular in Assumption (P) it is sufficient that $\gamma<1-\alpha -\frac\beta2-\frac\delta2$. Moreover the spectral dimension restriction $d_S<4$ can be lifted. Using the aforementioned corresponding results from \cite{HZ12} and the present methods one can get similar results as in the previous section, the proofs being completely analogous.\\

Moreover, one can easily consider linear combinations of noise terms such as $$\sum_{i=1}^N G_i(u)\cdot \dot z_i$$ for any finite integer $N$ and $G_i$ and $z_i$ as $G$ and $z$ in this paper.

\vskip3mm
Here we list a few examples. The results of Section \ref{subsc: regularity} can be applied to the same kind of stochastic equations considered in the paper by Hinz and Z\"ahle \cite[Section 7]{HZ12}. In particular, we can consider stochastic partial differential equations driven by fractional Brownian noises on metric measure spaces. The equations are studied in the pathwise sense and the results are valid $\mathbb P$-a.s.

\begin{example}
A classical example is the nonlinear heat equation on a smooth bounded domain $D\subset\mathbb R^n$, $n\le 3$, provided with the Lebesgue measure $\mu$, driven by fractional  Brownian field, see e.g.~\cite[Section 6]{HZ09_II}.
We consider a real valued fractional Brownian sheet $\{B^{H, K}(t,x), [0,t_0]\times \mathbb R^n\}$ with Hurst indices $0<H<1$ and $0<K<1$ for time and space respectively. This is a centered Gaussian field on $[0,t_0]\times \mathbb R^n$ with stationary rectangular increments satisfying
$$E\big( B^{H,K}(s,x)-B^{H,K}(s,y)-B^{H,K}(t,x)+B^{H,K}(t,y)\big)^2= c |s-t|^{2H} |x-y|^{2K}.$$
It can be shown \cite[Section 6]{HZ09_II} that there exits a version such that for almost all trajectories $\omega\in\Omega$ one has $B^{H,K}(\omega)\in C^{1-\alpha}([0,t_0], H^\sigma_q(\mu))$ for $0<1-\alpha<H, \, 0<\sigma<K$ and $1<q<\infty$.
Thus the distributional spatial partial derivatives   $\frac{\partial}{\partial x_i}B^{H,K}$ for $i=1, \ldots,n$ belong to $C^{1-\alpha}([0,t_0], H^{-\beta}_q(\mu))$, where $-\beta=\sigma-1$. Let  $F:\mathbb R^n \to \mathbb R$ and $G:\mathbb R^n \to \mathbb R^n$ such that $F$ and each component   $G_i$  satisfy Assumption (FG). Then Theorem \ref{thm: gamma holder regularity} and Corollary \ref{cor: extra holder regularity} can be applied in the pathwise sense to
\begin{equation}\label{eq: example in Rn}
\left\{
\begin{array}{l}
\frac{\partial u}{\partial t} = \Delta_D u +F(u) +\left\langle G(u), \frac\partial{\partial t}\nabla{B}^{H,K}\right\rangle\\
 u(0,x)=0, \textrm{ for } t\in(0,t_0) \\
 u(t,x)=0, \textrm{ for } x\in \partial D
\end{array}
\right.
\end{equation}
for almost all paths. Here $\Delta_D$ is the classical Dirichlet Laplacian, i.e., $-\Delta_D$ generates a heat semigroup  on $L_2(D)$ with Gaussian estimates. $\nabla B^{H,K}$ denotes the distributional gradient of the fractional Brownian sheet. The term $\left\langle G(u), \frac\partial{\partial t}\nabla{B}^{H,K}\right\rangle$ in \eqref{eq: example in Rn} is  given by $$\sum_{i=1}^n  G_i(u) \frac{\partial^2}{ \partial t \partial x_i} B^{H,K}\, , $$
where $\frac{\partial^2}{ \partial t \partial x_i} B^{H,K}$ is interpreted as $\dot{z}_i$ in the above sense.\\
Since the spectral dimension of $\mathbb R^n$ is $d_S=n$, for almost all sample paths the unique solution is an element of $$W^\gamma([0,t_0], H^\delta_\infty(\mu))\cap C^\gamma([0,t_0], H^\delta(\mu))$$ for all $\gamma$ and $\delta$ such that $1-H<\gamma<H-\frac{1-K}{2}-\frac{n}{4}$ and $1-K<\delta<\min (\frac{n}{2}, 1)$.
This can be satisfied only if $n\le3$. \\According to Remark \ref{rem: low spectral dim} for $d_S=n=1$ the result remains valid under the parameter condition \eqref{eq: parameter low spectral dim}, since the multiplication property \eqref{eq: product assumption} is fulfilled in this Euclidean case. The latter can be seen as follows. Theorem 4.5.2 in Runst and Sickel \cite{RS96} and restriction to functions vanishing on $\partial D$ lead to the estimate
$$\|uv\|_{H^\beta_p(\mu)}\le c\|u\|_\delta\|v\|_\beta\ ,~~~~
 1<p<2\, ,~~~~ \frac{1}{2}\le\beta<\delta<1\ .$$
Then for these parameters \eqref{eq: product assumption} follows by duality arguments. Hence, for $n=1$ and $K\le\frac{1}{2}$  H\"older continuous solutions to equation \eqref{eq: introduction: mild sol}  can be obtained if $H>\frac{2-K}{2}$.
Note that $K=\frac{1}{2}$ means white noise in space.
\end{example}

\begin{example}
A more sophisticated case mentioned in \cite[Example 2]{HZ12}) is the following. Let $(X,d,\mu)$ be a compact metric measure space satisfying Assumption (MMS) and admitting a semigroup $\{T(t), t\geq 0\}$ generated by a (fractal) Neumann Laplacian $\Delta$ associated to a local regular Dirichlet form $(\mathcal E, D(\mathcal E))$ on $X$, i.e., $-A=\Delta$. For various classes of fractals the corresponding heat kernels exist and satisfy Assumption (HKE($\beta$)) for any $\beta>0$ (see, e.g., Barlow and Bass \cite{BB92} and \cite{BB99}, Barlow, Bass, Kumagai and Teplyaev \cite{BBKT10}, Fitzsimmons, Hambly and Kumagai \cite{FHK94}, Hambly and Kumagai \cite{HK99}, Kigami \cite{Ki12}, Barlow, Grigor'yan and Kumagai \cite{BGK12} and the references therein).\\

A standard example for the noise process $z$, modified for our situation, is the following:  Let $e_0,e_1,e_2,\dots$ be a complete orthonormal system of eigenfunctions of $A$ in $L_2(\mu)$ and $\lambda_i$ be the corresponding eigenvalues,
 $B_1(t)^H,B_2(t)^H,\dots$ are i.i.d.\ fractional Brownian motions in $\mathbb{R}$ with Hurst exponent $\frac{1}{2}<H<1$,
and consider the formal series
$$b^H(t)=\sum_{i=1}^\infty B_i^H(t)\, q_i\, e_i$$

for real coefficients $q_i$. Then we get a modification $b^H$ such that a.s.
$$z:=b^H\in C^{1-\alpha}\left([0,t_0], H_q^{-\beta}(\mu)\right)$$
(with convergence of the series in these spaces) for any $0<1-\alpha<H$,
$\beta_*<\beta<1$ and $q>2$ under the following conditions on the measure $\mu$, the parameter $\beta_*$, the eigenfunctions $e_i$ and the coefficients $q_i$ for $i\geq 1$:\\
(a) $||e_i||_\infty\le c_1 \lambda_i^{a_1/2}$ and \\
(b) $|e_i(x)-e_i(y)|\le c_2 \lambda_i^{a_2/2}d(x,y)^b$ (up to an exceptional set)\\
for some positive constants $a_1, a_2, b, c_0, c_1, c_2$, and for $a:=\max(a_1,a_2)$,
$$\sum_{i=1}^\infty q_i^2\, \lambda_i^{-\beta_*+a-b2/w} <\infty \, .$$
(Note that in the case $q=2$, which is not relevant for our purposes, Conditions (a) and (b) are not needed and the convergence $\sum_{i=1}^\infty q_i^2\, \lambda_i^{-\beta_*} <\infty$ would be sufficient for the above property of $z$.)
 \vskip3mm ({\it Idea of proof:}
By the mapping properties of the resolvent operators $J^\sigma(\mu)$, which may be replaced here by $I^\sigma(\mu):=(-\Delta)^{\sigma/2}$ since the included eigenvalues are strictly positive, it is equivalent to get a modification, which satisfies a.s.
$$I^{\beta_*+b2/w}(\mu)b^H\in C^{1-\alpha}([0,t_0], H_q^{-\beta+\beta_*+b2/w}(\mu))$$
for any $q$ and $\beta>\beta_*$. For $0<\delta'<\delta$ the embedding of the H\"older space $C^\delta(X)$ into $H_q^{\delta'2/w}(\mu)$ can be seen, e.g., from the arguments in the proof of \cite[Proposition 5.6]{HZ09} (using there only the upper estimates taking into regard that the lower heat kernel estimates imply $\mu(B(x,r))\le c_0 r^{d_H}$ for any ball with centre $x$ and radius $r$). Therefore a sufficient condition for the above one is that $$I^{\beta_*+b2/w}(\mu)b^H\in C^{1-\alpha}([0,t_0], C^\delta(X))$$
for any $\delta<b$, and the latter can be proved by means of the Kolmogorov principle for the random function $$Y(t,x):=I^{\beta_*+b2/w}b^H(t,x)=\sum_{i=1}^\infty B_i^H(t)\, q_i\,\lambda_i^{-(\beta_*+b2/w)/2}\, e_i(x)\, .$$
We get
\begin{eqnarray*}
& &\mathbb{E}\big(Y(t_1,x_1)-Y(t_2,x_1)-(Y(t_1,x_2)-Y(t_2,x_2)\big)^2\\
&=& \sum_{i=1}^\infty q_i^2\lambda_i^{-(\beta_*+b2/w)}\mathbb{E}(B^H_i(t_1)-B^H_i(t_2))^2(e_i(x_1)-e_i(x_2))^2\\
&\le& \sum_{i=1}^\infty q_i^2\lambda_i^{-(\beta_*+b2/w)}|t_1-t_2|^{2H}c_2\lambda_i^a\,d(x_1,x_2)^{2b}\\
&\le& c\,|t_1-t_2|^{2H}d(x_1,x_2)^{2b}\, .
\end{eqnarray*}
Moreover,
\begin{eqnarray*}
& &\mathbb{E}(Y(t_1,x)-Y(t_2,x))^2\\
&=& \sum_{i=1}^\infty q_i^2\lambda_i^{-(\beta_*+b2/w)}\mathbb{E}(B^H_i(t_1)-B^H_i(t_2))^2|e_i(x)|^2\\
&\le& \sum_{i=1}^\infty q_i^2\lambda_i^{-(\beta_*+b2/w)}|t_1-t_2|^{2H}c_1^2\lambda_i^a\\
&\le& c\,|t_1-t_2|^{2H}\, .
\end{eqnarray*}
Using that the higher moments of centered Gaussian random variables are powers of the second moments this ensures the usual construction of a modification of $Y$ with the desired H\"older regularity by means of an extension of the values on a countable dense subset of $[0,t_0]\times X$.)
\vskip3mm
Note that because of the above ultracontractivity of the semigroup Condition (a) is always fulfilled for $a_1:=\frac{d_S}{2}$. Furthermore, if we work with the resistance metric $R(x,y)$ w.r.t.\ the Dirichlet form $\mathcal E$ then Condition (b) is satisfied for $a_2=b=1$.\\
For p.c.f.\ fractals with regular harmonic structures we have $d_S=\frac{2d_H}{d_H+1}<2$, see Kigami \cite{Ki01}. Moreover, under some mild additional assumptions on such fractals in Euclidean spaces the resistance metric $R$ satisfies $R(x,y)\le|x-y|^b$ for some $b> 0$, see Hu and Wang \cite{HW06}. Hence, in this case Condition (b) is also fulfilled for the Euclidean metric.\\
Examples with spectral dimension greater than 2 are provided by generalized Sierpinski carpets, see Barlow and Bass \cite{BB99}, or by certain products of fractals, see Strichartz \cite{St05}.
\vskip 3mm
According to Theorem \ref{thm: gamma holder regularity} function solutions to Equation \eqref{eq: introduction: mild sol} which are H\"older regular in time can be found for $d_S= \frac{2d_H }{w}<4$ (recall that $d_H$ denotes the Hausdorff dimension of $X$ and $w$ the walk dimension of the semigroup) and Hurst exponent
$H>\frac{1}{2}+\frac{d_S}{8}$. Recall that in this case we have $\beta<\frac{d_S}{2}$.\\
If $\frac{d_S}{2}\le 1$ then Remark \ref{rem: low spectral dim} provides the alternative parameter condition $\frac{1}{2}\le\beta< 2H-1$ for existence of H\"older continuous solutions.
In particular, for $\beta=\frac{1}{2}$ the noise is ``white'' in space.\\
For the classical case of the Dirichlet Laplace operator on the unit interval this example has been treated in Gubinelly, Lejay and Tindel \cite{GLT06} with different methods.\\

\end{example}

%\begin{acknowledgements}
%If you'd like to thank anyone, place your comments here
%and remove the percent signs.
%\end{acknowledgements}

% BibTeX users please use one of
%\bibliographystyle{spbasic}      % basic style, author-year citations
%\bibliographystyle{spmpsci}      % mathematics and physical sciences
%\bibliographystyle{spphys}       % APS-like style for physics
%\bibliography{}   % name your BibTeX data base

% Non-BibTeX users please use

%
% and use \bibitem to create references. Consult the Instructions
% for authors for reference list style.
%
%%%%%%%%%%%%%%%%%%%% BIBLIOGRAFIA %%%%%%%%%%%%%%%%%%%%%%%%%%%

\end{document}